\documentclass[11pt]{amsart}

\usepackage[T1]{fontenc}
\usepackage{amsmath,amssymb,amsthm,mathtools}
\usepackage{microtype}
\usepackage[hidelinks]{hyperref}
\usepackage[margin=1in]{geometry}

\hypersetup{
  pdftitle={The Minimum Cardinality of a Dependent Finite Gabor System Is Four},
  pdfauthor={Xinan Dai, Wenhao Deng, Yingdong Shi, Tailin Wu, Yuchen Yang}
}

\title{The Minimum Cardinality of a Dependent Finite Gabor System Is Four}

\author{Xinan Dai}
\thanks{Xinan Dai is currently a Ph.D. student at Fudan University and a visiting
student at the AI for Scientific Simulation and Discovery Lab,
Westlake University.}
\address{College of Future Information and Technology\\
Fudan University, Shanghai, China}
\curraddr{Department of Artificial Intelligence\\
School of Engineering\\
Westlake University, Hangzhou, China}
\email{xndai23@m.fudan.edu.cn}

\author{Wenhao Deng}
\thanks{Wenhao Deng is a student at the University of Glasgow and is currently an intern at the AI for Scientific Simulation and Discovery Lab, Westlake University.}
\address{University of Glasgow, Glasgow, United Kingdom}
\curraddr{Department of Artificial Intelligence\\
School of Engineering\\
Westlake University, Hangzhou, China}
\email{dengwenhao@westlake.edu.cn}

\author{Yingdong Shi}
\address{School of Information Science and Technology, ShanghaiTech University, Shanghai, China}
\email{shiyd2023@shanghaitech.edu.cn}

\author{Tailin Wu}
\address{Department of Artificial Intelligence\\
School of Engineering\\
Westlake University, Hangzhou, China}
\email{wutailin@westlake.edu.cn}

\author{Yuchen Yang}
\address{Department of Artificial Intelligence\\
School of Engineering\\
Westlake University, Hangzhou, China}
\email{yangyuchen@westlake.edu.cn}

\date{August 8, 2026}

\newtheorem{theorem}{Theorem}[section]
\newtheorem{proposition}[theorem]{Proposition}
\newtheorem{lemma}[theorem]{Lemma}
\newtheorem{corollary}[theorem]{Corollary}
\theoremstyle{remark}
\newtheorem{remark}[theorem]{Remark}

\newcommand{\R}{\mathbb R}
\newcommand{\C}{\mathbb C}
\newcommand{\Z}{\mathbb Z}
\newcommand{\Q}{\mathbb Q}
\newcommand{\T}{\mathbb T}
\newcommand{\cS}{\mathcal S}
\newcommand{\cZ}{\mathcal Z_2}
\newcommand{\wind}{\operatorname{wind}}
\newcommand{\tr}{\operatorname{tr}}
\newcommand{\adj}{\operatorname{adj}}
\newcommand{\diag}{\operatorname{diag}}

\begin{document}

\begin{abstract}
Recent work produced a linearly dependent system of twelve time--frequency
shifts of a Schwartz function, disproving the HRT conjecture.  We show that
four shifts already suffice, and hence that four is the smallest possible
cardinality of a dependent finite Gabor system.  More precisely, set
\[
 \alpha=\frac13+10^{-12}\sqrt2,\qquad
 \beta =\frac13+10^{-12}\sqrt3.
\]
We construct a nonzero complex-valued function $f\in\cS(\R)$ and
$\lambda\ne0$ such that
\[
 \left(I+\frac12W(1,0)+\frac12W(0,1/2)\right)
 W(\alpha,\beta/2)f=\lambda f,
\]
where $W$ denotes the Weyl time--frequency shift.  Since every system of at
most three shifts of a nonzero $L^2(\R)$ function is linearly independent,
this gives the sharp cardinality threshold.

The construction uses the rank-two Zak bundle naturally associated with the
covolume-$1/2$ lattice generated by $(1,0)$ and $(0,1/2)$.  At the rational
translation $(1/3,1/3)$, the three-step return has a uniformly dominated
contracting line.  A finite outward-rounded interval certificate proves that
this line is topologically trivial.  A quantitative perturbation argument
carries the dominated line to the explicit algebraic translation above.  A
winding calculation and a Diophantine cohomological equation then flatten its
scalar multiplier, and inverse Zak folding produces the required Schwartz
function.
\end{abstract}

\maketitle

\section{Introduction}\label{sec:intro}

For $(a,b)\in\R^2$, let
\begin{equation}\label{eq:Weyl-def}
 (W(a,b)f)(t)=e^{2\pi i b(t-a/2)}f(t-a),\qquad f\in L^2(\R).
\end{equation}
The operators $W(a,b)$ are the Weyl form of the Schr\"odinger representation;
our phase convention gives
\begin{equation}\label{eq:weyl-law}
 W(a,b)W(a',b')
 =e^{\pi i(ba'-ab')}W(a+a',b+b').
\end{equation}
See, for example, \cite[Chapter~1]{Folland89} for the Weyl calculus and its
basic covariance properties.

Heil, Ramanathan, and Topiwala asked whether, for every nonzero
$f\in L^2(\R)$ and every finite set of distinct points
$\Lambda\subset\R^2$, the family
\[
 \{W(z)f:z\in\Lambda\}
\]
is linearly independent \cite[p.~2790]{HRT96}.  The question became known as
the HRT conjecture.  It is elementary for two points, and the original paper
proved the three-point case; see \cite[Theorem~1, pp.~2791--2792]{HRT96}.
Linnell subsequently proved independence whenever the finite configuration is
contained in a discrete subgroup of the time--frequency plane
\cite[Proposition~1.3, p.~3270]{Linnell99}.  A substantial literature then
established further cases by imposing geometric or analytic structure: among
others, special four-point configurations were treated by Demeter
\cite{Demeter10} and Demeter--Zaharescu \cite{DemeterZaharescu12}, while strong
decay or asymptotic hypotheses on the window were studied in
\cite{BenedettoBourouihiya15,BownikSpeegle13}.  For later developments and
different reduction principles, see \cite{Grochenig15,Heil06,Liu19,Okoudjou19}.

The global conjecture is now known to be false.  Faulhuber, Petersen,
van Velthoven, and Voigtlaender recently constructed a Schwartz function with
twelve linearly dependent time--frequency shifts \cite[Theorem~1.1]{FPVV26}.  Once
existence of dependence is known, a natural finite question remains:
\emph{how many shifts are actually needed?}  Define
\[
 N_*:=\min\bigl\{|\Lambda|:\ \exists\,0\ne f\in L^2(\R)
 \text{ with }\{W(z)f:z\in\Lambda\}\text{ dependent}\bigr\}.
\]
The three-point theorem gives $N_*\ge4$, while the twelve-point example gives
$N_*\le12$.  Our main result closes this gap.

\begin{theorem}[explicit four-shift dependence]\label{thm:main}
Set
\begin{equation}\label{eq:explicit-tau}
 \alpha=\frac13+10^{-12}\sqrt2,
 \qquad
 \beta=\frac13+10^{-12}\sqrt3.
\end{equation}
There exist a nonzero complex-valued $f\in\cS(\R)$ and a scalar
$\lambda\in\C^*$ such that
\begin{equation}\label{eq:main-eigen}
 \left(I+\frac12W(1,0)+\frac12W(0,1/2)\right)
 W(\alpha,\beta/2)f=\lambda f.
\end{equation}
\end{theorem}

By \eqref{eq:weyl-law}, equation \eqref{eq:main-eigen} is the nontrivial
four-term relation
\begin{align}\label{eq:four-relation}
 -\lambda f
 &+W(\alpha,\beta/2)f
 +\frac12e^{-\pi i\beta/2}W(\alpha+1,\beta/2)f \notag\\
 &+\frac12e^{\pi i\alpha/2}W(\alpha,\beta/2+1/2)f=0.
\end{align}
All four coefficients are nonzero and the four phase-space points are
distinct.  Thus Theorem~\ref{thm:main}, together with three-point
independence, gives the exact boundary.

\begin{corollary}\label{cor:boundary}
The minimum cardinality of a dependent finite Gabor system generated by a
nonzero $L^2(\R)$ function is
\[
 \boxed{N_*=4}.
\]
The same minimum is obtained if the window is required to belong to
$\cS(\R)$.
\end{corollary}

The complex nature of the example matters.  Guan and Okoudjou have proved
four-point independence for every real-valued $L^2(\R)$ window
\cite[Corollary~1]{GuanOkoudjou26}, so the function in Theorem~\ref{thm:main} cannot be
chosen real-valued.  The configuration also lies outside the large-covolume
regime recently certified by Oussa \cite[Theorem~2.1]{Oussa26}: after translating one
point to the origin, the two lattice vectors are $(1,0)$ and $(0,1/2)$, whose
symplectic covolume is $1/2$, whereas the theorem in \cite{Oussa26} assumes
covolume strictly larger than one in the corresponding irrational case.

The proof is guided by this covolume-$1/2$ geometry.  Folding the scalar Zak
transform along the second frequency coordinate produces a rank-two vector
bundle.  In that bundle, the three-term operator
\[
 I+\frac12W(1,0)+\frac12W(0,1/2)
\]
becomes an everywhere invertible $2\times2$ matrix field.  Thus dependence is
not created by a zero of a scalar trigonometric polynomial.  Instead, the
relevant object is an invariant complex line for a matrix cocycle.

There is a useful tension between the rational and irrational parameters.
At the rational translation
\[
 \theta=(1/3,1/3),
\]
the base dynamics closes after three steps, which makes the return matrix
explicit enough to analyze.  Linnell's theorem prevents this rational model
itself from producing a counterexample, since the resulting configuration is
lattice-contained.  The rational model nevertheless supplies the geometry:
its return has a dominated, topologically trivial contracting line.  Moving
to the nearby algebraic translation \eqref{eq:explicit-tau} preserves this
line while breaking the lattice obstruction.  The arithmetic of
$\Q(\sqrt2,\sqrt3)$ then provides precisely the polynomial small-divisor bound
needed to solve a smooth torus cohomology equation.  In short, the rational
parameter supplies the invariant geometry and the irrational perturbation
supplies the cohomological flexibility.

The paper is organized accordingly.  Section~\ref{sec:zak} gives the rank-two
Zak reduction.  Section~\ref{sec:return} analyzes the rational return and
states the finite certificate used to trivialize its contracting line.
Section~\ref{sec:perturb} makes the continuation quantitative and proves that
the concrete parameter \eqref{eq:explicit-tau} is admissible.
Section~\ref{sec:cohomology} carries out the winding and Diophantine
cohomology argument, and Section~\ref{sec:inverse} unfolds the resulting
section back to a Schwartz function.  The appendix records only the numerical
inequalities that enter the proof.

\section{The rank-two Zak reduction}\label{sec:zak}

We use the scalar Zak transform
\begin{equation}\label{eq:zak}
 (Zf)(x,\nu)=\sum_{k\in\Z}f(x-k)e^{2\pi i k\nu}.
\end{equation}
For standard facts about the Zak transform and its unitary extension, see
\cite[Chapter~8, pp.~147--174]{Grochenig01}.  It satisfies
\begin{equation}\label{eq:zak-seams}
 Zf(x+1,\nu)=e^{2\pi i\nu}Zf(x,\nu),
 \qquad
 Zf(x,\nu+1)=Zf(x,\nu).
\end{equation}
Fold one frequency period into two components by
\begin{equation}\label{eq:vector-zak}
 (\cZ f)(x,\omega)=2^{-1/2}
 \begin{pmatrix}
  Zf(x,\omega/2)\\
  Zf(x,(\omega+1)/2)
 \end{pmatrix}.
\end{equation}
This is the two-component vector-Zak transform used in the recent
counterexample construction of Faulhuber, Petersen, van Velthoven, and
Voigtlaender; see \cite[Section~3.2]{FPVV26}.  We keep the sewing matrices
explicit because their noncommuting return holonomy is essential in the
four-point construction below.  The map $\cZ$ is unitary from $L^2(\R)$
onto the measurable sections of a rank-two bundle $E\to\T^2$ with sewing
matrices
\begin{equation}\label{eq:seams}
 S_1(\omega)=e^{\pi i\omega}Z_0,
 \qquad
 S_2=X_0,
 \qquad
 Z_0=\begin{pmatrix}1&0\\0&-1\end{pmatrix},
 \quad
 X_0=\begin{pmatrix}0&1\\1&0\end{pmatrix}.
\end{equation}
Thus a lifted section $F:\R^2\to\C^2$ represents a section of $E$ exactly
when
\begin{equation}\label{eq:section-seams}
 F(x+1,\omega)=S_1(\omega)F(x,\omega),
 \qquad
 F(x,\omega+1)=S_2F(x,\omega).
\end{equation}

For $m,n\in\Z$, direct substitution in \eqref{eq:zak} gives
\begin{equation}\label{eq:Lmn}
 \cZ W(m,n/2)\cZ^{-1}=L_{m,n}(x,\omega),
\end{equation}
where
\begin{equation}\label{eq:Lmn-formula}
 L_{m,n}=e^{\pi i(nx-m\omega+mn/2)}
 \begin{cases}
 \diag(1,e^{-\pi im}),&n\equiv0\pmod2,\\[1mm]
 \begin{pmatrix}0&1\\e^{-\pi im}&0\end{pmatrix},&n\equiv1\pmod2.
 \end{cases}
\end{equation}
Consequently the three-term lattice polynomial in
\eqref{eq:main-eigen} becomes multiplication by
\begin{equation}\label{eq:A}
 A(x,\omega)=I+\frac12e^{-\pi i\omega}Z_0
                +\frac12e^{\pi ix}X_0.
\end{equation}
Writing
\[
 u=e^{\pi ix},\qquad v=e^{-\pi i\omega},
\]
we have
\begin{equation}\label{eq:detA}
 \det A=1-\frac{u^2+v^2}{4},
 \qquad
 |\det A|\ge\frac12.
\end{equation}
Thus the lattice polynomial is everywhere invertible on the rank-two Zak
bundle.

Let $\tau=(\alpha,\beta)$ and put
\[
 T_\tau z=z-\tau\quad(z\in\T^2).
\]
The Zak covariance of $W(\alpha,\beta/2)$ gives the bundle map
\begin{equation}\label{eq:Ctau}
 C_\tau(x,\omega)
 =e^{\pi i\beta(x-\alpha/2)}A(x,\omega):E_{T_\tau z}\longrightarrow E_z.
\end{equation}
The identities
\[
 A(-u,v)=Z_0A(u,v)Z_0,
 \qquad
 A(u,-v)=X_0A(u,v)X_0
\]
together with the scalar in \eqref{eq:Ctau} show that this is a genuine
bundle morphism across both seams.  Hence \eqref{eq:main-eigen} is equivalent
to finding a nonzero smooth section $F$ and $\lambda\in\C^*$ such that
\begin{equation}\label{eq:cocycle-eigen}
 C_\tau(z)F(T_\tau z)=\lambda F(z).
\end{equation}

\section{The rational return and its contracting line}\label{sec:return}

Set
\[
 \theta=(1/3,1/3)
\]
and consider the three-step cocycle
\begin{equation}\label{eq:Dtau}
 D_\tau(z)=C_\tau(z)C_\tau(T_\tau z)C_\tau(T_\tau^2z):
 E_{T_\tau^3z}\longrightarrow E_z.
\end{equation}
At $\tau=\theta$, the input fibre is over $z-(1,1)$.  The order of the two
sewing operations is essential: from \eqref{eq:section-seams},
\begin{equation}\label{eq:holonomy}
 F(x-1,\omega-1)=vX_0Z_0F(x,\omega).
\end{equation}
Put $r=e^{\pi i/3}$.  Removing unit scalar factors, which do not affect
projective dynamics, gives the projective return
\begin{equation}\label{eq:P}
 P(u,v)=A(u,v)A(ur^{-1},vr)A(ur^{-2},vr^2)X_0Z_0.
\end{equation}
Exact multiplication in
$\Q(r)[u^{\pm1},v^{\pm1}]$, with $r^2-r+1=0$, yields
\begin{align}
 \tr P&=(3/2-3r)uv,                                      \label{eq:trP}\\
 \det P&=1-\frac{v^6}{64}-\frac{3u^2v^2}{16}-\frac{u^6}{64}.
                                                               \label{eq:detP}
\end{align}
Restoring the scalar factors from the three Weyl steps and from
\eqref{eq:holonomy} gives the genuine bundle endomorphism
\begin{equation}\label{eq:R}
 R(u,v)=(-iuv)P(u,v).
\end{equation}
Thus
\begin{align}
 \tr R&=-i(3/2-3r)u^2v^2,                                 \label{eq:trR}\\
 \det R&=-u^2v^2
 \left(1-\frac{v^6}{64}-\frac{3u^2v^2}{16}-\frac{u^6}{64}\right).
                                                               \label{eq:detR}
\end{align}
The parenthesis in \eqref{eq:detR} lies in the disk of radius $7/32$ about
$1$, and is therefore nonzero and null-homotopic in $\C^*$.

\begin{lemma}[uniform return domination]\label{lem:domination}
The two eigenvalues of $R(z)$ have distinct moduli for every $z\in\T^2$, with
a uniform gap.  The smaller-modulus eigenvalue
$\lambda_s^{\mathrm{full}}$ is nowhere zero and has winding $(0,0)$ on
$\T^2$.  The contracting eigenline is the same for $R$ and $P$.
\end{lemma}

\begin{proof}
Since $|-iuv|=1$, the two matrices have the same eigenlines and the same ratio
of eigenvalue moduli.  From \eqref{eq:trP}--\eqref{eq:detP},
\[
 |\tr P|^2=\frac{27}{4},
 \qquad
 |\det P|\le1+\frac1{64}+\frac3{16}+\frac1{64}=\frac{39}{32}.
\]
If the two eigenvalues had the same modulus $\rho$, then
$|\tr P|^2\le4\rho^2=4|\det P|$, contrary to
$27/4>39/8$.  Compactness gives a uniform gap.

The smaller root of the characteristic polynomial of the full return is
\begin{equation}\label{eq:small-root}
 \lambda_s^{\mathrm{full}}
 =\frac{\det R}{\tr R}
   \frac{2}{1+\sqrt{1-4\det R/(\tr R)^2}},
\end{equation}
where the square root is the branch near $1$.  Indeed,
\[
 \left|\frac{4\det R}{(\tr R)^2}\right|
 =\left|\frac{4\det P}{(\tr P)^2}\right|
 \le\frac{13}{18}<1.
\]
Hence the square-root argument remains in the disk of radius $13/18$ about
$1$, and the second factor in \eqref{eq:small-root} is homotopic through
nonvanishing functions to $1$.  Equations \eqref{eq:trR}--\eqref{eq:detR}
show that $\tr R$ and $\det R$ both have winding $(1,-1)$.  Their quotient
therefore has winding zero, proving
$\wind(\lambda_s^{\mathrm{full}})=(0,0)$.
\end{proof}

Eigenvalue winding does not determine the topology of its eigenline.  We now
remove that separate obstruction.  On the fundamental square let
\begin{equation}\label{eq:chi}
 \chi(x,\omega)=2^{-1/2}
 \begin{pmatrix}s+cv\\s-cv\end{pmatrix},
 \qquad
 s=\sin(\pi x/2),\quad c=\cos(\pi x/2),
\end{equation}
and extend by the sewing rules.  Put
\begin{equation}\label{eq:hframe}
 h_0=L_{0,1}\chi,
 \qquad
 J=\adj P,
 \qquad
 h=Jh_0,
 \qquad
 n=(-\overline{h_2},\overline{h_1})^T.
\end{equation}
The section $h_0$ has unit norm and is nowhere vanishing.  Since $P$ is
invertible, $h$ is nowhere vanishing as well.

Here it is useful to distinguish once more between projective and genuine
returns.  In dimension two, $\adj(cP)=c\,\adj(P)$.  By \eqref{eq:R},
\begin{equation}\label{eq:genuine-H}
 \C\,\adj(P)h_0=\C\,\adj(R)h_0.
\end{equation}
Thus the line
\[
 H:=\C h
\]
is the image of the trivial line $\C h_0$ under the inverse of the genuine
bundle automorphism $R$, up to a nowhere-zero scalar.  In particular, $H$ is
a globally defined trivial line bundle.  This is the role of the frame
$h,n$; the matrix $P$ is used only because it is the scalar-normalized
representative entering the projective calculation.

Rather than normalizing the equal-length frame, set
\begin{align}\label{eq:kij-def}
 k_{11}&=\langle h,Jh\rangle, &
 k_{12}&=\langle h,Jn\rangle,\notag\\
 k_{21}&=\langle n,Jh\rangle, &
 k_{22}&=\langle n,Jn\rangle.
\end{align}
Since $h$ and $n$ are orthogonal and have the same norm, the coordinate matrix
of $J$ in the frame $(h,n)$ is $\|h\|^{-2}(k_{ij})$.  The common scalar is
irrelevant to projective dynamics.  Thus the induced action on a graph
coordinate is
\begin{equation}\label{eq:Phi}
 \Phi_z(\xi)=\frac{k_{21}+k_{22}\xi}{k_{11}+k_{12}\xi}.
\end{equation}

\begin{proposition}[finite graph certificate]\label{prop:certificate}
For every $z\in\T^2$ and every $|\xi|\le1/4$,
\begin{equation}\label{eq:graph-cert}
 |\Phi_z(\xi)|<\frac14,
 \qquad
 |\partial_\xi\Phi_z(\xi)|<1.
\end{equation}
More precisely, the outward-rounded certificate gives
\begin{equation}\label{eq:cert-numbers}
 d\ge 15.4044994528652104323,
 \quad
 b\le0.2499996482177990124,
 \quad
 c\le0.7725133354867000155,
\end{equation}
where
\begin{align*}
 d&=|k_{11}|-\tfrac14|k_{12}|,\\
 b&=\frac{|k_{21}|+\frac14|k_{22}|}{d},\\
 c&=\frac{|k_{11}k_{22}-k_{12}k_{21}|}{d^2}.
\end{align*}
Consequently the contracting eigenline $L_\theta$ of $R$ is a graph over
$H$, and is therefore topologically trivial.
\end{proposition}

\begin{proof}
The three inequalities in \eqref{eq:cert-numbers} imply that the closed graph
disk $|\xi|\le1/4$ is mapped strictly into itself by the projective action of
$J$, with a uniform contraction factor strictly below one.  The contraction
mapping theorem gives a unique invariant graph.  Since $J$ is projectively
$P^{-1}$, this graph is precisely the eigenline corresponding to the
smaller-modulus eigenvalue of $P$, equivalently of $R$.  Projection along the
second frame direction identifies this graph with the trivial line $H$.
\end{proof}

\section{Quantitative continuation to an explicit irrational translation}\label{sec:perturb}

The rational parameter is used only to reveal the invariant geometry.  We now
show quantitatively that this geometry survives on a small, explicit
parameter square.  This is the step that turns the existence statement into
the concrete choice \eqref{eq:explicit-tau}.

Write
\[
 \tau=\theta+\Delta,
 \qquad
 \Delta=(\delta_1,\delta_2),
 \qquad
 \delta=\|\Delta\|_\infty.
\]
Since
$3\tau=(1,1)+3\Delta$, the three-step return maps
$E_{z-(1,1)-3\Delta}$ to $E_z$.  After applying the fixed $(1,1)$ sewing
identification, and again discarding nonzero scalar factors, its projective
representative is
\begin{equation}\label{eq:Ptau}
 P_\tau(z)=A(z)A(z-\tau)A(z-2\tau)X_0Z_0:
 E_{z-3\Delta}\longrightarrow E_z.
\end{equation}
Let $J_\tau=\adj P_\tau$, viewed projectively as the inverse return.  The
fixed rational frame $h,n$ from \eqref{eq:hframe} is evaluated at the source
and target points.  With $w=z-3\Delta$, define
\begin{align}\label{eq:pert-k}
 k_{11}^\tau&=\langle h(w),J_\tau(z)h(z)\rangle,&
 k_{12}^\tau&=\langle h(w),J_\tau(z)n(z)\rangle,\notag\\
 k_{21}^\tau&=\langle n(w),J_\tau(z)h(z)\rangle,&
 k_{22}^\tau&=\langle n(w),J_\tau(z)n(z)\rangle.
\end{align}
Because $h(w)$ and $n(w)$ are orthogonal and have the same norm, the common
normalizing factor cancels from the corresponding projective coordinate.
At $\tau=\theta$ these are exactly the coefficients in
Proposition~\ref{prop:certificate}.

\begin{proposition}[an explicit continuation square]\label{prop:explicit-square}
If
\begin{equation}\label{eq:explicit-square}
 \|\tau-\theta\|_\infty\le 2\cdot10^{-12},
\end{equation}
then the inverse three-step projective return maps the closed graph disk
$|\xi|\le1/4$ strictly into itself and is a strict contraction there.  It has
a unique invariant graph $L_\tau$.  The line bundle $L_\tau$ is smooth,
topologically trivial, depends continuously on $\tau$, and satisfies both
\begin{align}
 D_\tau(z)L_\tau(T_\tau^3z)&=L_\tau(z),                 \label{eq:three-inv}\\
 C_\tau(z)L_\tau(T_\tau z)&=L_\tau(z).                 \label{eq:one-inv}
\end{align}
In particular, the parameter \eqref{eq:explicit-tau} is admissible.
\end{proposition}

\begin{proof}
We make the perturbation estimate deliberately coarse.  On the unit torus,
\begin{equation}\label{eq:A-norms}
 \|A\|\le2,
 \qquad
 \|\partial_xA\|,\ \|\partial_\omega A\|\le\frac\pi2.
\end{equation}
Only the second and third factors of \eqref{eq:Ptau} move with $\tau$.
A telescoping product estimate therefore gives
\begin{equation}\label{eq:P-pert}
 \|P_\tau(z)-P_\theta(z)\|
 \le12\pi\delta<40\delta.
\end{equation}
For $2\times2$ matrices, adjugation is linear in the entries and preserves
the spectral norm up to the fixed unitary transpose identification.  Hence
\begin{equation}\label{eq:J-bounds}
 \|J_\tau-J_\theta\|<40\delta,
 \qquad
 \|J_\tau\|\le8.
\end{equation}
The explicit unit section $h_0$ satisfies
\[
 \|h_0\|=1,
 \qquad
 \|\partial_xh_0\|\le\frac{3\pi}{2},
 \qquad
 \|\partial_\omega h_0\|\le\pi.
\]
Differentiating the rational product gives
$\|\partial_xJ_\theta\|,\|\partial_\omega J_\theta\|\le6\pi$.  Since
$h=J_\theta h_0$, it follows that
\begin{equation}\label{eq:h-bounds}
 \|h\|=\|n\|\le8,
 \qquad
 \|\partial_xh\|,\ \|\partial_\omega h\|<60,
\end{equation}
and the same derivative bound holds for $n$.  Since
$\|w-z\|_\infty\le3\delta$,
\begin{equation}\label{eq:frame-shift}
 \|h(w)-h(z)\|,\ \|n(w)-n(z)\|<400\delta.
\end{equation}
Combining \eqref{eq:J-bounds}--\eqref{eq:frame-shift} in
\eqref{eq:pert-k} gives, uniformly in $z$,
\begin{equation}\label{eq:k-perturbation}
 |k_{ij}^\tau-k_{ij}^\theta|<30000\delta
 \qquad(1\le i,j\le2).
\end{equation}
Under \eqref{eq:explicit-square} we may therefore set
\[
 \eta:=6\cdot10^{-8}
\]
and use $|k_{ij}^\tau-k_{ij}^\theta|\le\eta$.

Let $r_0=1/4$, and denote the three rational certificate bounds in
\eqref{eq:cert-numbers} by
\[
 D=15.4044994528652104323,
 \quad B=0.2499996482177990124,
 \quad C=0.7725133354867000155.
\]
If
\[
 d_\tau=|k_{11}^\tau|-r_0|k_{12}^\tau|,
 \qquad
 N_\tau=|k_{21}^\tau|+r_0|k_{22}^\tau|,
\]
then
\begin{align*}
 d_\tau&\ge d_\theta-(1+r_0)\eta,\\
 N_\tau&\le N_\theta+(1+r_0)\eta.
\end{align*}
At the rational parameter,
$N_\theta\le B d_\theta$ and $d_\theta\ge D$.  Hence
\begin{align}\label{eq:image-margin}
 r_0d_\tau-N_\tau
 &\ge (r_0-B)D-(1+r_0)^2\eta\\
 &>5.3\cdot10^{-6}>0.\notag
\end{align}
Thus the perturbed projective map still sends the closed radius-$1/4$ graph
disk strictly into itself.

It remains to retain contraction.  From \eqref{eq:h-bounds} and
\eqref{eq:J-bounds}, the rational coefficients satisfy
$|k_{ij}^\theta|\le512$.  Consequently
\begin{equation}\label{eq:detK-pert}
 \left|
 (k_{11}^\tau k_{22}^\tau-k_{12}^\tau k_{21}^\tau)
 -(k_{11}^\theta k_{22}^\theta-k_{12}^\theta k_{21}^\theta)
 \right|
 \le2048\eta+2\eta^2.
\end{equation}
Moreover
$d_\tau\ge D-(1+r_0)\eta>15.4044$.  Using
$|\det(k_{ij}^\theta)|\le C d_\theta^2$ and
$d_\theta\le d_\tau+(1+r_0)\eta$, equations
\eqref{eq:detK-pert} and \eqref{eq:cert-numbers} give
\begin{equation}\label{eq:pert-contraction}
 \sup_{z,\,|\xi|\le1/4}|\partial_\xi\Phi_{\tau,z}(\xi)|<0.773<1.
\end{equation}
This proves a uniform contraction on the whole parameter square.

The contraction mapping theorem now gives a unique continuous invariant graph
$L_\tau$, and the usual fixed-point estimate gives continuous dependence on
$\tau$.  Projection of the graph onto $H$ is a bundle isomorphism, so
$L_\tau$ is trivial.  Smoothness in $z$ follows directly by differentiating
the graph iteration: at derivative order $j$, the only term containing the
$j$th derivative of the previous iterate is multiplied by
$\partial_\xi\Phi_{\tau,z}$, whose norm is uniformly $<0.773$; all remaining
terms contain only lower derivatives.  Induction in $j$ therefore gives
uniform convergence of every derivative and hence $L_\tau\in C^\infty$.

For a two-dimensional invertible cocycle, uniform contraction of the inverse
projective graph transform on an invariant graph is the standard graph-transform
criterion for a dominated splitting; the invariant graph is its unique stable
line.  Hence the fixed graph above is the stable line of a dominated splitting
for the three-step cocycle.  The cocycle identity
\begin{equation}\label{eq:cyclic}
 D_\tau(z)C_\tau(T_\tau^3z)=C_\tau(z)D_\tau(T_\tau z)
\end{equation}
shows that
\[
 \widetilde L_\tau(z):=C_\tau(z)L_\tau(T_\tau z)
\]
is another stable line for $D_\tau$.  Indeed, iterating
\eqref{eq:cyclic} conjugates the three-step dynamics on
$L_\tau(T_\tau z)$ to that on $\widetilde L_\tau(z)$ by the bounded invertible
map $C_\tau$.  The stable line of a dominated splitting is unique, so
$\widetilde L_\tau=L_\tau$.  This is \eqref{eq:one-inv}.

Finally, for \eqref{eq:explicit-tau},
\[
 \|\tau-\theta\|_\infty
 =10^{-12}\max\{\sqrt2,\sqrt3\}
 <2\cdot10^{-12},
\]
which proves admissibility of the announced parameter.
\end{proof}

\begin{remark}\label{rem:why-explicit}
The scale $10^{-12}$ is not intended to be optimal.  It is chosen so that a
short perturbative estimate fits comfortably inside the already certified
rational margins.  No additional parameter-space interval computation is
needed for the explicit example.
\end{remark}

\section{Winding and the smooth cohomology equation}\label{sec:cohomology}

For $\tau$ in the square \eqref{eq:explicit-square}, write the invariant line
as a graph over $H$.  Since $h$ is a nowhere-zero section of $H$, there is a
unique nowhere-zero smooth section $v_\tau$ of $L_\tau$ whose projection onto
$H$ is $h$.  The graph construction makes
\[
 (\tau,z)\longmapsto v_\tau(z)
\]
continuous and smooth in $z$.  By the one-step invariance
\eqref{eq:one-inv}, there is a unique smooth nonvanishing scalar function
$q_\tau:\T^2\to\C^*$ satisfying
\begin{equation}\label{eq:q}
 C_\tau(z)v_\tau(T_\tau z)=q_\tau(z)v_\tau(z).
\end{equation}
The family $(\tau,z)\mapsto q_\tau(z)$ is continuous.  For instance, with a
fixed Hermitian metric,
\[
 q_\tau(z)=
 \frac{\langle C_\tau(z)v_\tau(T_\tau z),v_\tau(z)\rangle}
      {\|v_\tau(z)\|^2}.
\]

At the rational parameter, three iterations of \eqref{eq:q} give the genuine
full-return identity
\begin{equation}\label{eq:q-product}
 \lambda_s^{\mathrm{full}}(z)
 =q_\theta(z)q_\theta(T_\theta z)q_\theta(T_\theta^2z).
\end{equation}
Translation of a nonvanishing torus function does not change either winding
number.  Lemma~\ref{lem:domination} therefore gives
\[
 3\wind(q_\theta)=\wind(\lambda_s^{\mathrm{full}})=(0,0),
\]
and hence
\begin{equation}\label{eq:qtheta-winding}
 \wind(q_\theta)=(0,0).
\end{equation}
The segment joining $\theta$ to the explicit parameter
\eqref{eq:explicit-tau} lies entirely in \eqref{eq:explicit-square}.  Since
$q_\tau$ is a continuous $\C^*$-valued family, winding is constant along this
segment.  Thus the function $q=q_\tau$ at our explicit parameter also has
zero winding in both torus directions.

The remaining issue is arithmetic.  The purpose of the particular
perturbation by $\sqrt2$ and $\sqrt3$ is to obtain a polynomial lower bound
for the small divisors of the translation.

\begin{lemma}[Diophantine bound]\label{lem:dio}
For $\alpha,\beta$ in \eqref{eq:explicit-tau}, there is a constant $c>0$ such
that for every $(m,n)\in\Z^2\setminus\{0\}$,
\begin{equation}\label{eq:dio}
 \|m\alpha+n\beta\|_{\R/\Z}
 \ge c(|m|+|n|)^{-3}.
\end{equation}
\end{lemma}

\begin{proof}
Write $\varepsilon=10^{-12}=1/b$ with $b=10^{12}$.  Let $\ell$ be a nearest
integer to $m\alpha+n\beta$ and set $B=|m|+|n|\ge1$.  In the biquadratic field
$K=\Q(\sqrt2,\sqrt3)$, the number
\begin{equation}\label{eq:zeta}
 \zeta:=3b(m\alpha+n\beta-\ell)
 =b(m+n-3\ell)+3m\sqrt2+3n\sqrt3
\end{equation}
is an algebraic integer.  It is nonzero because
$1,\sqrt2,\sqrt3$ are linearly independent over $\Q$.  Therefore
$|N_{K/\Q}(\zeta)|\ge1$.

Since $\ell$ is a nearest integer to $m\alpha+n\beta$, one has
$|\ell|=O(B)$.  Under each of the other three real embeddings of $K$, the
conjugate of \eqref{eq:zeta} is therefore $O(B)$.  Taking the norm gives
\[
 1\le |N_{K/\Q}(\zeta)|\le |\zeta|\,C B^3
\]
for a constant $C$ depending only on the fixed parameter.  Division by $3b$
yields \eqref{eq:dio}.
\end{proof}

Because $q:\T^2\to\C^*$ has zero winding on both fundamental cycles, it has a
smooth logarithm:
\begin{equation}\label{eq:logq}
 q=e^\phi,
 \qquad
 \phi\in C^\infty(\T^2;\C).
\end{equation}
Let
\[
 \mu=\int_{\T^2}\phi(z)\,dz.
\]
For $k=(k_1,k_2)\in\Z^2\setminus\{0\}$ define
\begin{equation}\label{eq:gfourier}
 \widehat g(k)
 =\frac{\widehat\phi(k)}{1-e^{-2\pi i k\cdot\tau}},
 \qquad
 \widehat g(0)=0.
\end{equation}
The elementary estimate
$|1-e^{-2\pi it}|\ge4\|t\|_{\R/\Z}$ and
Lemma~\ref{lem:dio} show that the denominator in \eqref{eq:gfourier} has only
polynomially growing inverse.  Since the Fourier coefficients of $\phi$
decay faster than any power, so do those of $g$.  Thus
$g\in C^\infty(\T^2)$ and
\begin{equation}\label{eq:cohomology}
 g(z)-g(T_\tau z)=\phi(z)-\mu.
\end{equation}
Putting
\begin{equation}\label{eq:F}
 F=e^g v_\tau
\end{equation}
and using \eqref{eq:q}, \eqref{eq:logq}, and \eqref{eq:cohomology}, we obtain
\begin{equation}\label{eq:flat}
 C_\tau(z)F(T_\tau z)=e^\mu F(z).
\end{equation}
Thus the scalar multiplier of the invariant line has been flattened to the
constant $\lambda=e^\mu\ne0$.

\section{Unfolding the Zak section}\label{sec:inverse}

Write $F=(F_1,F_2)^T$.  Reverse the folding in \eqref{eq:vector-zak} by
setting, for one frequency period,
\begin{equation}\label{eq:unfold}
 \widetilde F(x,\nu)=
 \begin{cases}
  \sqrt2\,F_1(x,2\nu),&0\le\nu<1/2,\\
  \sqrt2\,F_2(x,2\nu-1),&1/2\le\nu\le1.
 \end{cases}
\end{equation}
The $S_2=X_0$ sewing in \eqref{eq:seams} identifies the two pieces, including
all frequency derivatives, at $\nu=1/2$.  The $S_1$ sewing gives
\begin{equation}\label{eq:scalar-seams-back}
 \widetilde F(x+1,\nu)=e^{2\pi i\nu}\widetilde F(x,\nu),
 \qquad
 \widetilde F(x,\nu+1)=\widetilde F(x,\nu).
\end{equation}
Hence $\widetilde F$ is a smooth scalar Zak section.

For $x\in[0,1)$ and $k\in\Z$, define
\begin{equation}\label{eq:inverse-zak}
 f(x-k)=\int_0^1\widetilde F(x,\nu)e^{-2\pi i k\nu}\,d\nu.
\end{equation}
The first identity in \eqref{eq:scalar-seams-back} makes the definitions on
adjacent unit intervals agree, together with all $x$-derivatives.  For any
nonnegative integers $M,j$, differentiate \eqref{eq:inverse-zak} $j$ times in
$x$ and integrate by parts $M$ times in $\nu$.  The second identity in
\eqref{eq:scalar-seams-back} cancels all boundary terms, giving
\[
 |k|^M|f^{(j)}(x-k)|\le C_{M,j}
\]
uniformly for $x\in[0,1)$.  These are exactly the Schwartz seminorm bounds,
so
\[
 f\in\cS(\R).
\]
The vector Zak transform is unitary and $F$ is nowhere zero, hence $f\ne0$.
Finally, the covariance formulas \eqref{eq:Lmn} and \eqref{eq:Ctau} turn
\eqref{eq:flat} into \eqref{eq:main-eigen}.  This proves
Theorem~\ref{thm:main}.

\begin{proof}[Proof of Corollary~\ref{cor:boundary}]
Theorem~1 of Heil--Ramanathan--Topiwala gives linear independence for every
configuration of at most three distinct points
\cite[pp.~2791--2792]{HRT96}, so $N_*\ge4$.  Theorem~\ref{thm:main} gives a
four-point Schwartz example, so $N_*\le4$.  The same argument proves the
Schwartz-space statement.
\end{proof}

\begin{remark}[What the construction isolates]\label{rem:meaning}
The matrix field $A$ is uniformly invertible by \eqref{eq:detA}.  The
four-point dependence is therefore not explained by a singularity of the
three-term lattice polynomial.  What changes at covolume $1/2$ is the rank of
the natural Zak model: the scalar obstruction is replaced by a rank-two
bundle cocycle, and dependence is encoded by a topologically trivial invariant
line whose scalar dynamics can be cohomologically flattened.  The rational
three-cycle is not itself a counterexample; it is a finite model from which
the invariant geometry can be read and then transported to an irrational
translation.  This separation between geometric construction and arithmetic flattening
suggests a possible route to other subcritical configurations.
\end{remark}

\appendix
\section{The finite certificate}\label{app:certificate}

Only Proposition~\ref{prop:certificate} uses computer-assisted inequalities.
The calculation is a rigorous interval computation, not a sampled numerical
experiment.  It evaluates the explicit formulas
\eqref{eq:A}, \eqref{eq:P}, \eqref{eq:chi}, and \eqref{eq:hframe} using
Arb/Acb midpoint-radius arithmetic with outward error bounds; see
\cite{Johansson17} for the underlying arithmetic model.

The square $[0,1]^2$ is covered initially by a $128\times128$ rational grid.
A box is accepted only when interval evaluation proves simultaneously
\begin{align}\label{eq:app-cert}
 |k_{11}|-\frac14|k_{12}|&>0,\notag\\
 \frac{|k_{21}|+\frac14|k_{22}|}
 {|k_{11}|-\frac14|k_{12}|}&<\frac14,\\
 \frac{|k_{11}k_{22}-k_{12}k_{21}|}
 {( |k_{11}|-\frac14|k_{12}|)^2}&<1.\notag
\end{align}
A failing box is split dyadically, with at most four refinement levels.  At
160-bit precision all boxes are certified.  The resulting global outward
bounds are exactly those recorded in \eqref{eq:cert-numbers}; the computation
uses $1{,}043{,}952$ evaluated boxes and $787{,}060$ certified leaf boxes.
No pointwise floating-point sample enters the proof.

The Laurent identities \eqref{eq:trP}--\eqref{eq:detP}, the noncommuting seam
\eqref{eq:holonomy}, and the restoration of the full return
\eqref{eq:R} are algebraic identities independent of the interval output.
They may be checked exactly over $\Q(r)[u^{\pm1},v^{\pm1}]$ with
$r^2-r+1=0$.  As a negative control on the bundle identification, replacing
the required $X_0Z_0$ holonomy by the identity fails the graph inequalities.
The certificate therefore tests the actual sewn return rather than a matrix
living only on a fundamental square.

\end{document}